\documentclass[12pt]{amsart}

\usepackage{amsmath}
\usepackage{amsthm}
\usepackage{amsfonts}
\usepackage{hyperref}

\title[Rigidity on the moment curve and the conic]
{Bar-and-joint rigidity on the moment curve coincides with cofactor rigidity on a conic}

\author{Luis Crespo Ruiz \and Francisco Santos}

\address{
Departamento de Matem\'aticas, Estad\'istica y Computaci\'on,
Universidad de Cantabria,
39005 Santander, Spain
}
\email{francisco.santos@unican.es, luis.cresporuiz@unican.es}

\thanks{Supported by grants PID2019-106188GB-I00 funded by MCIN/AEI/10.13039/501100011033 and by FPU19/04163 of the Spanish Government and by project CLaPPo (21.SI03.64658) of Universidad de Cantabria and Banco Santander}

\newcommand{\rank}{\operatorname{rank}}
\newcommand{\R}{\mathbb{R}}
\newcommand{\RP}{\mathbb{RP}}
\newcommand{\RR}{\mathcal{R}}
\newcommand{\CC}{\mathcal{C}}
\newcommand{\HH}{\mathcal{H}}
\newcommand{\MM}{\mathcal{M}}
\newcommand{\NN}{\mathcal{N}}
\newcommand{\PP}{\mathcal{P}}
\newcommand{\JJ}{\mathcal{S}}
\newcommand{\p}{\mathbf{p}}
\newcommand{\x}{\mathbf{x}}
\newcommand{\ab}{\mathbf{a}}
\newcommand{\bb}{\mathbf{b}}
\newcommand{\q}{\mathbf{q}}

\newtheorem{theorem}{Theorem}[section]
\newtheorem{lemma}[theorem]{Lemma}
\newtheorem{proposition}[theorem]{Proposition}
\newtheorem{corollary}[theorem]{Corollary}
\newtheorem{conjecture}[theorem]{Conjecture}
\newtheorem{question}[theorem]{Question}
\theoremstyle{definition}
\newtheorem{definition}[theorem]{Definition}

\newtheorem{remark}[theorem]{Remark}

\begin{document}

\maketitle


\begin{abstract}
  We show that, for points along the moment curve, the bar-and-joint rigidity matroid and the hyperconnectivity matroid coincide, and that both coincide with the $C^{d-2}_{d-1}$-cofactor rigidity of points along any (non-degenerate) conic in the plane. 
  For hyperconnectivity in dimension two, having the points in the moment curve is no loss of generality.
  
  We also show that, restricted to bipartite graphs, the bar-and-joint rigidity matroid is freer than the hyperconnectivity matroid.
\end{abstract}



\section{Introduction}
\label{sec:intro}
Rigidity matroids were introduced by Graver in~\cite{Graver} as a way to formalize and generalize rigidity theory. Besides the matroids associated to (infinitesimal) \emph{bar-and-joint rigidity} they include two other important examples: the matroids of \emph{$C_{d-1}^{d-2}$-cofactor rigidity} introduced by Whiteley~\cite{Whiteley} and of the \emph{hyperconnectivity} introduced by Kalai~\cite{Kalai}. 
See the precise definitions below.

We show an interesting case in which the three theories are equivalent:

\begin{theorem}
	\label{thm:main}
	Let $d$ be a positive integer and let $t_1,\ldots,t_n\in \R$ be (distinct) real numbers. Then, the following three $d$-dimensional rigidity matroids coincide:
	\begin{itemize}
		\item The $C_{d-1}^{d-2}$-cofactor matroid of points $\{(t_i,t_i^2)\}_{i=1}^n$ along the standard parabola.
		\item The bar-and-joint matroid of points $\{(t_i,\dots,t_i^d)\}_{i=1}^n$ along the moment curve.
		\item The $d$-hyperconnectivity matroid of vectors $\{(1, t_i,\dots,t_i^{d-1})\}_{i=1}^n$, which, if no $t_i$ equals zero, is equivalent to that of $\{(t_i,\dots,t_i^{d})\}_{i=1}^n$.
	\end{itemize}
\end{theorem}

The \emph{moment curve} $\{(t,\dots,t^d): t\in \R\} \subset \R^d$ is an important curve from several points of view. From the perspective of algebraic geometry, it is the rational normal curve of degree $d$ (in a specific embedding, but any other embedding is algebraically isomorphic to it). From the perspective of polytope theory, any finite subset of $n$ points in the moment curve defines a \emph{cyclic polytope}, an example of a \emph{neighborly polytope}: a $d$-polytope with $n$ vertices that attains the maximum possible number of faces of each dimension.
From the perspective of rigidity theory, one important property of the moment curve is that it is contained in many (linearly independent) quadrics. By a result of Bolker and Roth~\cite{BR} this has drastic consequences for the bar-and-joint rigidity of bipartite graphs embedded with vertices in it. See Lemma~\ref{lemma:complete_bipartite} and Remark~\ref{rem:BolkerRoth}.

For $d=2$ the moment curve is the standard parabola.
Since the three forms of rigidity are projectively invariant, any non-degenerate conic can be substituted for the standard parabola, for arbitrary $d$  in part (1), or for $d=2$ in parts (2) and (3). 

\medskip
Let us  recall some of the needed rigidity concepts. Good comprehensive references  for background are \cite{NSW, Whiteley}; \cite{CJT} also contains everything we need. For hyperconnectivity see~\cite{Kalai} or~\cite{JT}.

A \emph{point configuration} is a finite list of distinct points $\p=(\p_1,\dots,\p_n)$ in $\R^d$. A \emph{framework} on $\p$ is a simple graph with vertex set $\{\p_1,\dots,\p_n\}$. Since the points in $\p$ are labelled by $[n]$, we identify frameworks on $\p$ with subsets of $\binom{[n]}{2}$. The \emph{bar-and-joint (infinitesimal) rigidity matrix} of $\p$ is the following $\binom{n}{2} \times nd$ matrix:
\[
R(\p):=
\begin{pmatrix}
\p_1-\p_2 & \p_2-\p_1 & 0 & \dots & 0 & 0 \\
\p_1-\p_3 & 0 & \p_3-\p_1 & \dots & 0 & 0 \\
\vdots & \vdots & \vdots & & \vdots & \vdots \\
\p_1-\p_n & 0 & 0 & \dots & 0 & \p_n-\p_1 \\
0 & \p_2-\p_3 & \p_3-\p_2 & \dots & 0 & 0 \\
\vdots & \vdots & \vdots & & \vdots & \vdots \\
0 & 0 & 0 & \dots & \p_{n-1}-\p_n & \p_n-\p_{n-1}
\end{pmatrix}
\]
That is to say, each row $r_{ij}$ of $R(\p)$ ($\{i,j\} \in \binom{[n]}{2}$) consists of $n$ blocks of size $d$ and has zeroes everywhere except in the $i$th and $j$th blocks, where it has, respectively, $\p_i-\p_j$ and $\p_j-\p_i$. 

\begin{definition}
	\label{defi:R}
	The \emph{bar-and-joint rigidity matroid} of $\p$ is the linear matroid of rows of $R(\p)$, with ground set $\binom{[n]}{2}$. We denote it $\RR_d(\p)$. A framework on $\p$ is called \emph{rigid} (respectively \emph{stress-free}, \emph{isostatic}) if the corresponding set of rows is spanning, (respectively \emph{independent}, \emph{a basis}) in the rigidity matroid. 
\end{definition}

The reason for this terminology is that, if $R(\p)|_G$ denotes the row-submatrix of $R(\p)$ corresponding to a framework $G$, then the kernel of $R(\p)|_G$ are the infinitesimal motions of the points (``joints'') that preserve lengths of all edges in $G$ (``bars''); the kernel of its transpose are the equilibrium stresses on $G$:  assignments of extension/contraction forces to the edges that cancel out at every vertex, so that the system is in equilibrium.

We can also give up symmetry and directly define the following \emph{hyperconnectivity matrix} of $\p$, introduced by Kalai~\cite{Kalai}:
\[
H(\p):=
\begin{pmatrix}
\p_2 & -\p_1 & 0 & \dots & 0 & 0 \\
\p_3 & 0 & -\p_1 & \dots & 0 & 0 \\
\vdots & \vdots & \vdots & & \vdots & \vdots \\
\p_n & 0 & 0 & \dots & 0 & -\p_1 \\
0 & \p_3 & -\p_2 & \dots & 0 & 0 \\
\vdots & \vdots & \vdots & & \vdots & \vdots \\
0 & 0 & 0 & \dots & \p_n & -\p_{n-1}
\end{pmatrix}.
\]

\begin{definition}
	\label{defi:H}
	We call \emph{hyperconnectivity matroid of $\p$}, and denote it $\HH_d(\p)$, the linear matroid of rows of $H(\p)$.
\end{definition}
For even $d$, $\HH_d(n)$ coincides with the algebraic matroid of skew-symmetric matrices of rank at most $d$, of interest in low-rank matrix completion problems. See, e.g.,~\cite{Bernstein}. This relation is well-known, but we offer a proof in  Proposition~\ref{prop:H_equals_S}. In the case $d=2$ this is, in turn, the same as the algebraic matroid of the Pl\"ucker embedding of the Grassmaniann $Gr(2,\R)$ in $\R^{\binom{n}2}$. The following combinatorial characterization of $\HH_2(n) $ is known~\cite{Bernstein}: a graph $G\subset\binom{[n]}2$ is independent in $\HH_2(n)$ if, and only if, $G$ admits an acyclic orientation with no alternating closed walk.

Finally, for a point configuration $\p\subset \R^2$ in the plane and a positive integer $d$ we define the  \emph{cofactor rigidity matrix} of degree $d-1$,
\[
\footnotesize
C_d(\p):=
\begin{pmatrix}
c(\p_1-\p_2) & -c(\p_1-\p_2) & 0 & \dots & 0 & 0 \\
c(\p_1-\p_3) & 0 & -c(\p_1-\p_3) & \dots & 0 & 0 \\
\vdots & \vdots & \vdots & & \vdots & \vdots \\
c(\p_1-\p_n) & 0 & 0 & \dots & 0 & -c(\p_1-\p_n) \\
0 & c(\p_2-\p_3) & -c(\p_2-\p_3) & \dots & 0 & 0 \\
\vdots & \vdots & \vdots & & \vdots & \vdots \\
0 & 0 & 0 & \dots & c(\p_{n-1}-\p_n) & -c(\p_{n-1}-\p_n)
\end{pmatrix},
\]
where  $ c(x,y)=(x^{d-1},x^{d-2}y,\ldots,xy^{d-2},y^{d-1})$.

\begin{definition}
	\label{defi:C}
	We call \emph{cofactor rigidity matroid of $\p$}, and denote it $\CC_{d-1}^{d-2}(\p)$, the linear matroid of rows of $C_d(\p)$.
\end{definition}

This formalism was introduced by Whiteley in the 1990's~\cite{Whiteley}, elaborating on previous work of Billera on smooth bivariate splines~\cite{Billera}.
The notation \emph{$\CC_{d-1}^{d-2}$} comes from the relation of this theory to $C^{d-2}$-splines of degree $d-1$.  

The three rigidity theories share important properties. 
Let $\p$ and $\q$ be two point configurations of size $n\ge d$ in general position in $\R^d$ and $\R^2$ respectively.\footnote{For $\RR_d$ and $\CC_{d-1}^{d-2}$ we here mean \emph{affine} general position; that is, all affinely dependent subsets are affinely spanning.
For $\HH_d$ we mean \emph{linear} general position: all linearly dependent subsets are linearly spanning.}
Then:
\begin{enumerate}
	\item $\rank(\RR_d(\p)) =\rank(\HH_d(\p)) = \rank(\CC_{d-1}^{d-2}(\q)) = nd-\binom{d+1}2$.
	\item Every copy of the complete graph $K_{d+2}$ is a circuit in $\RR_d(\p)$, $\HH_d(\p)$ and $\CC_{d-1}^{d-2}(\q)$. (And hence every copy of $K_m$, with $m\le d+1$, is independent).
\end{enumerate}
By a theorem of Nguyen~\cite{Nguyen}, matroids on the ground set $\binom{[n]}2$ and satisfying these two properties coincide exactly with the \emph{abstract rigidity matroids of dimension $d$} introduced by Graver~\cite{Graver}. 

For reasons that will become apparent in the proof, we call the matroid in Theorem \ref{thm:main} the \emph{polynomial rigidity matroid of degree $d-1$ with parameters $t_1,\dots, t_n$} (Definition~\ref{defi:poly}). We denote it $\PP_d(t_1,\dots, t_n)$. 
Since points on the moment curve are in general position, $\PP_d(t_1,\dots, t_n)$ is an abstract rigidity matroid too.

\medskip

There is some  interest among the experts on the relations between the three matroids $\HH_d(\p)$, $\RR_d(\p)$,  and $\CC_{d-1}^{d-2}(\q)$ in the case when $\p$ and $\q$ are generic. We denote these matroids $\HH_d(n)$, $\RR_d(n)$,  and $\CC_{d-1}^{d-2}(n)$, and denote by $\PP_d(n)$ the generic case of $\PP_d(t_1,\dots, t_n)$.

It is conjectured that $\CC_{d-1}^{d-2}(n)$ is freer than $\RR_d(n)$ and that the latter is freer than $\HH_d(n)$ (Conjecture~\ref{conj:freer}). It is also conjectured that $\CC_{d-1}^{d-2}(n)$ is the freest abstract rigidity matroid and that  $\HH_d(n)$ is the freest matroid  in which every $K_{d+2}$ and every $K_{d+1,d+1}$ are circuits (Conjectures~\ref{conj:maximalC} and ~\ref{conj:maximalH}). We recall what is known about these conjectures in Section~\ref{sec:discussion}, including a proof of the following result which is, to the best of our knowledge, new: restricted to bipartite graphs, $\HH_d(n)$ coincides with $\RR_d(\p)$ where $\p$ has the points of the two parts lying in two hyperplanes, but  otherwise generic. In particular, on bipartite graphs $\RR_d(n)$ is freer than $\HH_d(n)$  (Theorem \ref{thm:bipartite} and Corollary \ref{coro:bipartite}). 

We hope that our new matroid $\PP_d(n)$ may help shed light on these conjectures, since it shows a case where the three rigidity theories coincide.
We finish Section~\ref{sec:discussion} and the paper proving several properties that are common to  $\HH_d(n)$ and  $\PP_d(n)$ (Corollary~\ref{coro:complete_bipartite} and Propositions~\ref{prop:split} and~\ref{prop:diamond}). This suggests that perhaps $\HH_d(n) = \PP_d(n)$ (Question~\ref{question:equal}). For $d=2$ this is true, by the invariance of $\HH_d(n)$ under linear scaling of the positions of points (Theorem~\ref{thm:d=2}).

\section{Proof of Theorem~\ref{thm:main}}
\label{sec:proof}

The matroids $\CC_{d-1}^{d-2}$ and $\RR_d$ are invariant under projective transformation in $\RP^{d}\supset \R^d$, see e.g.~\cite{NSW}. $\HH_d$ is, as far as we know, only invariant under linear transformation or, rather, under projective transformation in $\RP^{d-1}$ as a quotient space of $\R^d\setminus\{0\}$:

\begin{lemma}
	\label{lemma:invariant}
	Let $\p=(\p_1,\dots,\p_n)$ and $\q=(\q_1,\dots,\q_n)$ be point configurations in $\R^d$ and assume that $\q$ is obtained from $\p$ by one of the following two procedures:
	\begin{enumerate}
		\item A global linear transformation $l:\R^d\to \R^d$ with $l(\p_i) = \q_i$ for all $i$.
		\item Multiplication of each $\p_i$ by a non-zero scalar $\alpha_i$, so that $\q_i=\alpha_i \p_i$ for each $i$.
	\end{enumerate}
	Then, $\HH_d(\p) = \HH_d(\q)$.
\end{lemma}

\begin{proof}
	In the first case, the matrix $H(\q)$ is obtained from $H(\p)$ multiplying on the right by a block-diagonal matrix having (the matrix of) $l$ in each block. In the second case, it is obtained by first multiplying each row $(i,j)$ by $\alpha_i \alpha_j$ and then dividing each column in block $i$ by $\alpha_i$. Multiplying rows or columns by non-zero scalars does not change the matroid.
\end{proof}

This implies that $\HH_2$  coincides with $\PP_2$:

\begin{theorem}
	\label{thm:d=2}
	Let $\p=(\p_1,\dots,\p_n)$ be points in $(\R\setminus \{0\})^2$, with $\p_i=(a_i,b_i)$. Let $t_i = b_i/a_i$ for $i=1,\dots,n$.
	Then, $\HH_2(\p)$ equals $\PP_2(t_1,\dots,t_n)$. In particular,  $\HH_2(n) = \PP_2(n)$. 
\end{theorem}
\begin{proof}
	Apply Lemma~\ref{lemma:invariant} with $\alpha_i = b_i/a_i^2$.
\end{proof}

Let $\R[t]_{< d}$ denote the $d$-dimensional vector space of univariate polynomials of degree less than $d$ with real coefficients.
For each choice of parameters $t_1,\dots,t_n$ and choice of one basis $F^{i}(t)=(f_1^{i}(t), \ldots, f_{d}^{i}(t))$ of $\R[t]_{< d}$ for each $i$, construct the following \emph{polynomial rigidity matrix of degree $d-1$}:
\[
P_{F^{1},\dots, F^{n}}(t_1,\dots,t_n):=
\begin{pmatrix}
F^{1}(t_2) & -F^{2}(t_1) & 0 & \dots & 0 & 0 \\
F^{1}(t_3) & 0 & -F^{3}(t_1) & \dots & 0 & 0 \\
\vdots & \vdots & \vdots & & \vdots & \vdots \\
F^{1}(t_n) & 0 & 0 & \dots & 0 & -F^{n}(t_1) \\
0 & F^{2}(t_3) & -F^{3}(t_2) & \dots & 0 & 0 \\
\vdots & \vdots & \vdots & & \vdots & \vdots \\
0 & 0 & 0 & \dots & F^{n-1}(t_n) & -F^{n}(t_{n-1})
\end{pmatrix}.
\]

\begin{lemma}
	\label{lemma:bases}
	The linear matroid of rows of $P_{F^{1},\dots, F^{n}}(t_1,\dots,t_n)$ is independent of the choice of bases $F^{1},\dots, F^{n}$.
\end{lemma}

\begin{proof}
	Given two choices of bases $(F^{1},\dots, F^{n})$ and $(G^{1},\dots, G^{n})$, the matrix corresponding to $G$ can be obtained from that of $F$ multiplying on the right by the  $nd\times nd$ block-diagonal matrix that has in the $i$-th diagonal block the $d\times d$ matrix that changes from basis $F^i$ to basis $G^i$.
\end{proof}

\begin{definition}
	\label{defi:poly}
	We denote this matroid $\PP_d(t_1,\dots,t_n)$ and call it the  \emph{polynomial rigidity matroid of degree $d-1$ with parameters $(t_1,\dots,t_n)$}.
\end{definition}

Theorem~\ref{thm:main} follows from the following more precise statement:

\begin{theorem}
	The three matroids of Theorem~\ref{thm:main} coincide with 
	$\PP_d(t_1,\dots,t_n)$.
\end{theorem}

\begin{proof}
	Consider the following bases of $\R[t]_{< d}$, depending on $t_i$:
	\[
	F^{t_i}(t)=\left(\frac{t_i^k-t^k}{t_i-t}\right)_{k=1,\dots,d},
	\qquad
	G^{t_i}(t)=
	\left((t_i+t)^{k-1}\right)_{k=1,\dots,d}.
	\]
	Then, $P_{F^{t_1},\dots, F^{t_n}}(t_1,\dots,t_n)$ equals the bar-and-joint rigidity matrix of points along the moment curve, except each row $(i,j)$ has been divided by $t_i-t_j$, which does not affect the matroid. Similarly, $P_{G^{t_1},\dots, G^{t_n}}(t_1,\dots,t_n)$ equals the cofactor rigidity matrix of points along the parabola, except each row $(i,j)$ has been divided by $(t_i-t_j)^{d-1}$. For the latter, observe that for points $(x_i,y_i)=(t_i, t_i^2)$
	along the parabola we have
	\[
	(x_i-x_j)^{d-k}(y_i-y_j)^{k-1}=
	(t_i-t_j)^{d-k}(t_i^2-t_j^2)^{k-1}=(t_i-t_j)^{d-1} (t_i+t_j)^{k-1}.
	\]
	Finally, the hyperconnectivity matrix of $\{(1,\dots,t_i^{d-1})\}_{i\in [n]}$ equals $P_{F,\dots, F}(t_1,\dots,t_n)$ where we choose all bases equal to the monomial basis $F=(1,t,t^2,\dots,t^{d-1})$.
\end{proof}

\section{Hyperconnectivity and low rank skew-symmetric matrices}
\label{sec:algebraic}

Recall that the \emph{algebraic matroid} of an irreducible variety $V\subset \R^N$  is the matroid with ground set $N$ and in which a subset of coordinates is independent if there is no (non-trivial) polynomial relation among them on $V$. Put differently, $X$ is independent if $I(V) \cap \R[X] = 0$. 
When $V$ is parametrized by a polynomial map $T:\R^M \to V\subset \R^N$, the algebraic matroid of $V$ equals the linear matroid of rows of the Jacobian of $T$ at a sufficiently generic point of $\R^M$~\cite[Proposition 2.5]{Rosen}. (It is important  that we are looking at algebraic varieties over $\R$. Algebraic matroids over finite fields are not all representable as linear matroids over those same  fields).

We are interested in the variety $S_d(n)\subset \R^{\binom{n}2}$ of skew-symmetric matrices of rank (at most) $d$, where each matrix is represented as the list of its entries with $1\le i < j \le n$. We denote
$\JJ_d(n)$ its algebraic matroid and we assume $d$ even, since every skew-symmetric matrix has even rank. 
The following statement seems to be known, but we have not found an explicit proof of it.

\begin{proposition}
	\label{prop:H_equals_S}
	For even $d$, $\HH_d(n) = \JJ_d(n)$.
\end{proposition}

\begin{proof}
	Let $d=2k$.
	Every skew-symmetric matrix of rank $\le2k$ can be expressed as 
	\[
	\sum_{l=1}^{k}(\ab_l^T \bb_l - \bb_l^T \ab_l)
	\] 
	for some $\ab_1,\bb_1, \dots, \ab_k,\bb_k\in \R^n$ (see, e.g., \cite[Lemma 1.3]{IMA} or \cite[Theorem 3.23]{PacStu}). 
	Thus, writing $\ab_l=(a_{l,1},\dots,a_{l,n})$ and $\bb_l=(b_{l,1},\dots,b_{l,n})$ we have a surjective map
	\begin{align}
		T: (\R^n)^{2k} &\to S_{2k}(n) \subset \R^{\binom{n}2} \nonumber\\
		(\ab_1,\bb_1,\dots,\ab_k,\bb_k) &\mapsto \sum_{l=1}^k \left(a_{l,i}b_{l,j}-a_{l,j}b_{l,i}\right)_{1\le i<j\le n}.
		\label{eq:T}
	\end{align}
	
	For each  $(\ab_l,\bb_l)_{l=1,\dots,k}\in (\R^n)^{2}$, we denote by $J((\ab_l,\bb_l)_{l=1,\dots,k})$ the Jacobian matrix of $T$ at that point; hence, $\JJ_{2k}(n)$ equals the linear matroid of rows of $J((\ab_l,\bb_l)_{l=1,\dots,k})$ for a sufficiently generic choice of $(\ab_l,\bb_l)_{l=1,\dots,k}$. 
	Now, from Eq.~\eqref{eq:T}
	it is straightforward to check that $J((\ab_l,\bb_l)_{l=1,\dots,k})$ is the hyperconnectivity matrix of $(\p_1,\dots,\p_n)$ where
	\[
	\p_i=(b_{1,i}, -a_{1,i},\dots, b_{k,i}, -a_{k,i}).
	\]
	Thus, $\HH_{2k}(n)$ also equals the linear matroid of rows of $J((\ab_l,\bb_l)_{l=1,\dots,k})$ for a sufficiently generic choice of $(\ab_l,\bb_l)_{l=1,\dots,k}$.
\end{proof}

Observe that for $d=2$ the map $T$ in the above proof is the Pl\"ucker embedding of $Gr(2,\R)$ into $\R^{\binom{n}2}$. Hence, $\HH_2(n)=\PP_2(n)$ is also the algebraic matroid of $Gr(2,\R)$.

\section{The generic case}
\label{sec:discussion}

\subsection{The three generic rigidity matroids}

Let $\RR_d(n)$, $\HH_d(n)$ and $\CC_{d-1}^{d-2}(n)$ denote the \emph{generic} bar-and-joint, hyperconnectivity  and  cofactor rigidity matroids of dimension $d$ for $n$ points; that is, the matroids $\RR_d(\p)$, $\HH_d(\p)$ and $\CC_{d-1}^{d-2}(\q)$ for  sufficiently generic choices of $\p\in (\R^{d})^n$ and $\q\in (\R^{2})^n$. 
There are the following  conjectures relating them:

\begin{conjecture}[\protect{\cite[Conjecture 11.5.1]{Whiteley}}]\ 
	\label{conj:maximalC}
	$\CC_{d-1}^{d-2}(n)$ is the  freest matroid in which  every  $K_{d+2}$ is a circuit.
	In particular, it is the freest among all abstract rigidity matroids of dimension $d$. 
\end{conjecture}

\begin{conjecture}[\protect{\cite[Conjecture~6.6(a)]{JT}}]\ 
	\label{conj:maximalH}
	$\HH_{d}(n)$ is the  freest matroid in which every $K_{d+2}$ and every $K_{d+1,d+1}$ are circuits.
\end{conjecture}

Here, we say that  \emph{$\NN$ is freer than $\MM$} and write $\MM \le \NN$, where $\NN$ and $\MM$ are two matroids on the same ground, if every independent set of $\MM$ is also independent in $\NN$. Equivalently, if the rank function of $\MM$ is bounded above by that of $\NN$. 

Of course, Conjecture \ref{conj:maximalC} would imply that $\CC_{d-1}^{d-2}(n)$ is freer than both of $\RR_d(n)$ and $\HH_d(n)$. In turn, a positive answer to \cite[Problem 3]{Kalai-shifting} would imply  $\RR_d(n)$ is freer than  $\HH_d(n)$ (see \cite[Section 2.7]{Kalai-shifting} for the reason of the implication), so we consider this a conjecture as well:

\begin{conjecture}[Kalai and Whiteley]
	\label{conj:freer}
	For every $n$ and $d$,
	\[
	\HH_d(n) \le \RR_d(n) \le \CC_{d-1}^{d-2}(n).
	\]
\end{conjecture}

All these  conjectures are trivial for $d=1$: the three matroids coincide with the graphic matroid, which is the only rigidity matroid in dimension one and satisfies that $K_{2,2}$ is a circuit. For higher dimensions the following is known:

For $d= 2$, Conjecture~\ref{conj:maximalH} is open (see partial results in \cite[Section 6.3.1]{JT}), but the rest hold true. Indeed,  Conjecture~\ref{conj:freer} follows from the more precise statement
\[
\HH_2(n) < \RR_2(n) = \CC_{1}^{0}(n).
\]
The equality on the right is trivial from the definitions, and the strict inequality on the left follows from combining Laman's characterization (which gives $\HH_2(n) \le \RR_2(n)$) with the fact that $K_{d+1,d+1}$ is a circuit in $\HH_d(n)$ for every $d$ while $K_{3,3}$ is a basis in $\RR_2(n)$. The Laman characterization implies also  Conjecture~\ref{conj:maximalC}.

In dimension $3$ or higher only the following is known:

\begin{itemize}
	
	\item  Conjecture~\ref{conj:maximalC} has recently been proven in dimension $3$ by Clinch, Jackson and Tanigawa~\cite{CJT},
	but is open in higher dimensions. 
	
	\item That result implies $\RR_3(n) \le \CC_{2}^{1}(n)$, and it is conjectured that $\RR_3(n) = \CC_{2}^{1}(n)$ \cite[Conjecture 10.3.2]{Whiteley}. In higher dimension we know that $\RR_d(n) \ne \CC_{d-1}^{d-2}(n)$, since  the $(d-4)$-th cone over $K_{6,7}$ is a basis in $\CC_{d-1}^{d-2}(n)$ and dependent in $\RR_d(n)$~\cite[Section 11]{Whiteley}, but  $\RR_d(n) \le \CC_{d-1}^{d-2}(n)$ is open.
	
	\item Whether $\HH_d(n)$ is less free than $\RR_d(n)$ and/or $\CC_{d-1}^{d-2}(n)$ is open even for $d=3$, but we do know graphs that are dependent in $\HH_d(n)$ and independent in the other two. For example,  no bipartite graph can be spanning in $\HH_d(n)$~(see Corollary~\ref{coro:complete_bipartite} below, which essentially follows \cite[Theorem 6.1]{Kalai}); hence,   $K_{d+1,\binom{d+1}2}$ is dependent in it, while it is a basis in  $\RR_d(n)$ and $\CC_{d-1}^{d-2}(n)$~\cite[Sect. 11]{Whiteley}.
	
\end{itemize}

A technique similar to the proof of Theorem~\ref{thm:main} gives us the following result, supporting the conjecture that $\HH_d(n) < \RR_d(n)$ for all $d\ge 3$:

\begin{theorem}
	\label{thm:bipartite}
	Let $V=X\dot\cup Y$ be a vertex set with a given bipartition and 
	let $\{\p_i : i\in V\} \subset \R^{d-1}$ be positions for the vertices. Then, restricted to bipartite graphs with that bipartition
	the following two $d$-dimensional rigidity matroids coincide:
	\begin{enumerate}
		\item the bar-and-joint  matroid of the points $\{(\p_i, 0): i\in X\} \cup \{(\p_j, 1):j\in Y\}$.
		\item the hyperconnectivity matroid of the points $\{(\p_i, 1): i\in V\}$.
	\end{enumerate}
\end{theorem}

\begin{proof} 
	Let $E=X\times Y$ be the edge set of the complete bipartite graph $K_{X\times Y}$. 

	For each point $\p_i$, choose a linear basis $F^i=(f^i_{1}, \dots, f^i_{d})$ of the space of affine functions $f:\R^{d-1}\to \R$.	
	For each such choice of bases $\mathcal F = (F^i)_{i\in V}$ we define the following \textit{affine rigidity matrix}  of $(\mathcal F ,\p)$, of size $|E|\times |V|d$:  in the row indexed by the edge $(i,j)\in E$, put the vector $F^i(\p_j)$ in block $i$ , the vector $-F^j(\p_i)$  in block $j$, and zeroes elsewhere.
	
	With the same argument of Lemma \ref{lemma:bases}, the linear matroid of rows of this matrix is independent of the choice of bases. Now, we have that:
	\begin{itemize}
		\item The choice  $F^i(\x) = (\x-\p_i,1)$ for $i\in X$ and $F^j(\x) = (\x-\p_j,-1)$ for $j\in Y$ makes the row of edge $(i,j)$ have the vector $(\p_j-\p_j,1)$ in block $i$ and the vector $(\p_i-\p_j,-1)$ in block $j$.
		\item The choice  $F^i(\x) = (\x,1)$ for every $i$ makes the row of edge $(i,j)$ have the vector $(\p_j, 1)$  in block $i$  and the vector $(-\p_i, -1)$  in block $j$.
		\qedhere
	\end{itemize}
\end{proof}

Observe that in this proof we only use bipartiteness for the construction in $\RR_d$, but not for $\HH_d$. That is, the \emph{affine rigidity matroid} of the point set $\{\p_1,\dots,\p_n\}$ constructed in the proof coincides with $\HH_d(\{(\p_i,1)\}_{i=1,\dots,n})$ on all graphs,  not only bypartite ones.

\begin{corollary}
	\label{coro:bipartite}
	Restricted to bipartite graphs we have that $\HH_d(n) \le \RR_d(n)$.
\end{corollary}

\begin{proof} Consider $\HH_d(n)$ realized with generic points $\q_1, \dots, \q_n \in \R^d$. By Lemma~\ref{lemma:invariant}, there is no loss of generality in assuming that they all have last coordinate equal to $1$, so we write $\q_i=(\p_i,1)$ with $\p_i\in \R^{d-1}$.  The previous theorem tells us that $\HH_d(n) = \HH_d(\q)$, restricted to bipartite graphs, coincides with $\RR_d$ on the point set  $\{(\p_i, 0): i\in X\} \cup \{(\p_j, 1):j\in Y\}$, which is less free than the generic $\RR_d(n)$.
\end{proof}

\subsection{Where does $\PP_d(n)$ lie.}

We can similarly consider the \emph{generic polynomial rigidity matroid of degree $d-1$}, that we denote $\PP_d(n)$, defined as the matroid $\PP_d(t_1,\dots,t_n)$ for a sufficiently generic choice of $(t_1,\dots, t_n)$.
It is obvious that $\HH_d(n)$, $\RR_d(n)$ and $\CC_{d-1}^{d-2}(n)$ are freer than $\PP_d(n)$, and the last two strictly so since $K_{\binom{d+1}{2},d+1}$ is a basis in both $\RR_{d}$ and $\CC_{d-1}^{d-2}$, but it is dependent in $\PP_d$ for $d\ge 2$ by 
Corollary~\ref{coro:complete_bipartite}  below. But we do not know of any graph that is independent in $\HH_d(n)$ and dependent in $\PP_d(n)$. It is thus plausible that the answer to the next question is positive:

\begin{question}
	\label{question:equal}
	Is
	$
	\PP_d(n) = \HH_d(n)
	$ 
	for every $n$ and $d$?
\end{question}

In this section we show some common properties of $\PP_d(n)$ and $\HH_d(n)$. 
Our first result is that every complete bipartite graph has the same rank in both. It follows from the following lemma, which generalizes \cite[Theorem 6.1]{Kalai} (Kalai considers only  generic positions).

\begin{lemma}
	\label{lemma:complete_bipartite}  
	Let $G=K_{n_1,n_2}$ be a complete bipartite graph, call $n=n_1+n_2$ and let $\p_1,\dots$, $\p_n\in \R^d$ be positions in linear general position (no $d$ of them lie in a linear hyperplane). Then, the rank of $G$ in $\HH_d(\p)$ equals:	
	\begin{enumerate}
		\item $n_1n_2$ (that is, $G$ is independent) if $\min\{n_1,n_2\} \le d$.
		\item $dn-d^2$ if $\min\{n_1,n_2\} \ge d$.
	\end{enumerate}
\end{lemma}

\begin{proof}
	In part (1) we have that all edges have an endpoint of degree at most $d$, and edges incident to an endpoint of such low degree are automatically independent of the rest as long as they are linearly independent among them, which in our case is guaranteed by general position.
	
	For part (2),
	an independent subgraph of $G$  with $dn-d^2$ edges can be obtained starting with a $K_{d,d}$ (independent by part (1)) and adding the rest of vertices one by one, connecting each new vertex to $d$ previously added vertices.
	
	As the rigidity matrix has $n_1n_2$ rows, to finish the proof we only need to show that the orthogonal complement of the column space has dimension at least  $(n_1-d)(n_2-d)$.  We consider columns as living in $\R^n\otimes \R^m$ in the natural way, with the entry in row $(i,j)$ representing the coordinate $e_i\otimes e_j$. Regarded in this way, every tensor product $l\wedge m$ of a linear dependence $l$ among the points in one part of $G(\p)$ and a linear dependence $m$ among the points in the other part lies in the orthogonal complement of the columns. By general position, we have $n_1-d$ and $n_2-d$ dimensions to choose $l$ and $m$ from, giving a tensor product space of dimension $(n_1-d)(n_2-d)$.
\end{proof}

\begin{remark}
	\label{rem:BolkerRoth}
	Observe that in part (1) general position is equivalent to ``$\p$ is affinely independent'', and in the proof of part (2) we only use that each part of $(G,\p)$ linearly spans $\R^d$, a condition weaker than general position. In fact, the same proof gives the more general result: the rank of $K_{n_1,n_2}$ in $\HH_d(\p)$ equals $n_1n_2 - (n_1-d_1)(n_2-d_2) $ where $d_1$ and $d_2$ are the dimensions of the linear spans of the points in the two parts of the framework $(K_{n_1,n_2},\p)$. 
	
	This result  is the analogue in $\HH_d$ of the main result of \cite{BR}. Although  \cite[Theorem 5.9]{BR} gives the  rank of $K_{n_1,n_2}$ in $\RR_d(\p)$ for any choice of $\p$, we here state only the cases analogue to the ones in Lemma~\ref{lemma:complete_bipartite}:
	\begin{enumerate}
		\item If $\min\{n_1,n_2\} \le d$ and each part is in general position (that is, affinely independent) then $K_{n_1,n_2}$ is independent in $\RR_d(\p)$; its rank equals $n_1n_2$.
		\item If $\min\{n_1,n_2\} \ge d$ and each part affinely spans $\R^d$ then the rank of $K_{n_1,n_2}$ in $\RR_d(\p)$ equals $nd-\binom{d+1}{2}$ minus the number of linearly independent quadrics containing $\p$. 
		That is: if no quadric contains $\p$ (for which a necessary condition is $n_1+n_2 \ge \binom{d+2}{2}$) then $K_{n_1,n_2}$ is spanning in $\RR_d(\p)$; from there each quadric containing $\p$ lowers rank by one.
	\end{enumerate}
	The $d$-dimensional moment curve is contained in $\binom{d}{2}$ quadrics: for each $2\le i\le d$   we have the quadric $x_i=x_1x_{i-1}$ ($d-1$ quadrics) and for each $2\le i \le j \le d-1$ we have the 
	quadric $x_1 x_{i+j-1} = x_i x_j$ if $i+j\le d+1$ or the quadric $x_d x_{i+j-d} = x_i x_j$ if $i+j > d+1$ ($\binom{d-1}{2}$ quadrics). Hence, the result of Bolker and Roth recovers that the generic rank of $K_{n_1,n_2}$ in $\PP_d(n)$ (that is, the rank in $\RR_d(\p)$ for $\p$ generic along the moment curve) equals
	\[
	nd-\binom{d+1}2 -\binom{d}2 = nd-d^2,
	\]
	as stated in Lemma~\ref{lemma:complete_bipartite}.
\end{remark}

\begin{corollary}
	\label{coro:complete_bipartite}  
	Every complete bipartite graph has the same rank in $\HH_d(n)$ and $\PP_d(n)$, given by the formulas in Lemma~\ref{lemma:complete_bipartite}. In particular, in these matroids no bipartite graph is spanning  and every $K_{d+1,d+1}$ is a circuit.
\end{corollary}

For our next result, recall that a \emph{vertex $d$-split}
on a graph $G$ is the following operation: choose a vertex $v$ of degree at least  $d-1$ with its neighbors divided in three parts $A$, $B$, and $C$ with $|B|=d-1$;  remove all edges of the form $(v,w)$ with $w\in C$ and insert a new vertex $v'$ with neighbors $B\cup C \cup\{v\}$.
Vertex $d$-splits are known to preserve independence in $\RR_d(n)$ and $\CC_{d-1}^{d-2}$. The same holds for $\PP_d(n)$ and $\HH_d(n)$ as we now show. We call corank of a subset $X$ of a matroid the difference $|X|-\rank(X)$.

\begin{proposition}
	\label{prop:split}
	Corank does not increase under vertex $d$-split neither in $\HH_d$ nor in $\PP_d$. In particular, vertex splits of independent graphs are independent.
\end{proposition}

\begin{proof}
	The standard proof that vertex split preserves independence in $\RR_d$~\cite{Whiteley-split} carries over as follows.
	To fix notation, let $B=\{v_1,\dots,v_{d-1}\}$ be the neighbors of $v$ that are joined to both $v$ and $v'$ after the split, and let $G'$ be the graph obtained by the split.
	
	Let $\p$ be arbitrary positions for the vertices of $G$, and denote $\p_w$ the position of each vertex $w$. 
	Let $\p'$ be positions for $G'$ defined by $\p'_w=\p_w$ for every vertex in $G$ and $\p'_{v'} = \p_v$.
	Assuming that $\p_v, \p_{v_1},\dots, \p_{v_{d-1}}$ are linearly independent (that is, a basis of $\R^d$) we
	have that every vector orthogonal to all rows of the matrix $H(\p)$ becomes orthogonal to all rows of $H(\p')$ by simply repeating in the block of $v'$  what was in the block of $v$. Conversely, every vector orthogonal to all rows of $H(\p')$ has the same content in the blocks of $v$ and $v'$. Indeed, if we call $u,u'\in \R^d$ what the vector has in the blocks of $v$ and $v'$, the vertex split construction implies that $u$ and $u'$ have the same scalar product with the vectors $\p_v, \p_{v_1},\dots,\p_{v_{d-1}}$.
	
	That is, as long as $\p_v, \p_{v_1},\dots, \p_{v_{d-1}}$ are a basis, the coranks of $G$ in $\HH_d(\p)$ and of $G'$ in $\HH_d(\p')$ coincide. In particular, this holds for generic positions (matroid $\HH_d(n)$) and for generic positions along the moment curve (matroid $\PP_d(n)$). Of course, the positions $\p'$ are not generic since $\p'_v = \p'_{v'}$, but perturbing $\p'$ to be generic can only increase the rank of $G'$ and decrease its corank.
\end{proof}

The following is an example where a vertex split makes the corank decrease in $\HH_2$: Let $G$ be the graph with edges $\{12, 23, 34, 14, 15, 25, 35, 46, 56\}$. It is a basis (in any abstract rigidity matroid of dimension $2$) since vertices $6$, $4$ and $3$, deleted in this order, remove  two edges each  and result in a $K_3$. Let $G'$ be the vertex split of $5$ in $G$ with $A=\{1,3\}$, $B=\{6\}$ and $C=\{2\}$. This is again a basis by the proposition above. It turns out that $G'$ is isomorphic to the vertex split of any vertex in $K_{3,3}$ with $|A|=|B|=|C|=1$; since $K_{3,3}$ is a circuit, this vertex split decreases the corank from one to zero.

A variation of vertex split is what we here call a \emph{diamond $d$-split}: choose again a vertex $v$, now of degree at least  $d$ and with its neighbors divided in three parts $A$, $B$, and $C$ with $|B|=d$. Then remove all edges of the form $(v,w)$ with $w\in C$ and insert a new vertex $v'$ with neighbors $B\cup C$. Observe that we now do not insert the edge $vv'$.
The following result is proved in \cite[Lemma 3.8]{KNN} in a context that includes $\HH_d$, but restricted to bipartite graphs:

\begin{proposition}
	\label{prop:diamond}
	Corank does not increase under diamond $d$-split neither in $\HH_d$ nor in $\PP_d$. In particular, diamond splits of independent graphs are independent.
\end{proposition}

\begin{proof}
	The proof is very similar to the previous one, and left to the reader. The only significant difference is that now we call $B=\{v_1,\dots,v_{d}\}$ and assume $\{\p_u: u\in B\}$ to be a linear basis, instead of $\{\p_u: u\in B\cup\{v\}\}$.
\end{proof}

Let us finish by summing up what these results give as for $d=2$, in which $\PP_2 =\HH_2$. In this case all bases of  $\HH_2$ are Laman graphs, and they include all \emph{planar} Laman graphs (since every planar Laman graph can be obtained from $K_3$ by vertex splits~\cite{FJW}), but no \emph{bipartite} Laman graph (since bipartite graphs cannot be rigid in $\HH_d$ for any $d\ge 2$).  
Theorem 6.9 in  \cite{JT}, taking into account that independence in $\HH_2$ is preserved under 0-extensions and diamond splits, implies that the case $d=2$ of Conjecture~\ref{conj:maximalH} holds if, and only if, it has the property that every dependent flat is a union of copies of $K_4$ and $K_{3,3}$.



\section*{Acknowledgements}

We thank Daniel Irving Bernstein, Bill Jackson, Gil Kalai, Meera Sitharam and Walter Whiteley for  comments on an early version of this paper. In particular, for directing us to  the possible relation between $\PP_d(n)$ and $\HH_d(n)$. We also thank Eran Nevo for clarifying discussions on the relation between rigidity and shifting, and  anonymous referees for helpful comments on the first version of this paper.


%
%

\bibliographystyle{alphaurl}
\bibliography{relation}

\begin{thebibliography}{rgomr10}

\bibitem[Ber17]{Bernstein}
Daniel~Irving Bernstein.
\newblock Completion of tree metrics and rank 2 matrices.
\newblock {\em Lin. Alg. App.}, 533(1):1--13, 11 2017.
\newblock \href {http://dx.doi.org/10.1016/J.LAA.2017.07.016}
  {\path{doi:10.1016/J.LAA.2017.07.016}}.

\bibitem[Bil88]{Billera}
Louis~J. Billera.
\newblock Homology of smooth splines: generic triangulations and a conjecture
  of {S}trang.
\newblock {\em Trans. Amer. Math. Soc.}, 310(1):325--340, 1988.
\newblock \href {http://dx.doi.org/10.2307/2001125}
  {\path{doi:10.2307/2001125}}.

\bibitem[BR80]{BR}
E.~D. Bolker and B.~Roth.
\newblock When is a bipartite graph a rigid framework?
\newblock {\em Pacific J. Math.}, 90(1):27--44, 1980.
\newblock URL: \url{http://projecteuclid.org/euclid.pjm/1102779115}.

\bibitem[CJT22]{CJT}
Katie Clinch, Bill Jackson, and Shin-ichi Tanigawa.
\newblock Abstract 3-rigidity and bivariate ${C}^1_2$-splines {I}: Whiteley's
  maximality conjecture.
\newblock {\em Disc. Analysis}, 2022(2), 2022.
\newblock \href {http://dx.doi.org/10.19086/da.34691}
  {\path{doi:10.19086/da.34691}}.

\bibitem[FJW04]{FJW}
Zsolt Fekete, Tibor Jord{\'a}n, and Walter Whiteley.
\newblock An inductive construction for plane {L}aman graphs via vertex
  splitting.
\newblock In Susanne Albers and Tomasz Radzik, editors, {\em Algorithms -- ESA
  2004}, pages 299--310, Berlin, Heidelberg, 2004. Springer Berlin Heidelberg.
\newblock \href {http://dx.doi.org/10.1007/978-3-540-30140-0_28}
  {\path{doi:10.1007/978-3-540-30140-0_28}}.

\bibitem[Gra91]{Graver}
Jack~E. Graver.
\newblock Rigidity matroids.
\newblock {\em SIAM J. Discrete Math.}, 4(3):355--368, 1991.
\newblock \href {http://dx.doi.org/10.1137/0404032}
  {\path{doi:10.1137/0404032}}.

\bibitem[JT21]{JT}
Bill Jackson and Shin-ichi Tanigawa.
\newblock Maximal matroids in weak order posets.
\newblock Preprint, 2021.
\newblock \href {http://arxiv.org/abs/2102.09901v2}
  {\path{arXiv:2102.09901v2}}.

\bibitem[Kal85]{Kalai}
Gil Kalai.
\newblock Hyperconnectivity of graphs.
\newblock {\em Graphs and Combinatorics}, 1:65--79, 1985.

\bibitem[Kal02]{Kalai-shifting}
Gil Kalai.
\newblock Algebraic shifting.
\newblock In {\em Computational commutative algebra and combinatorics ({O}saka,
  1999)}, volume~33 of {\em Adv. Stud. Pure Math.}, pages 121--163. Math. Soc.
  Japan, Tokyo, 2002.
\newblock \href {http://dx.doi.org/10.2969/aspm/03310121}
  {\path{doi:10.2969/aspm/03310121}}.

\bibitem[KNN16]{KNN}
Gil Kalai, Eran Nevo, and Isabella Novik.
\newblock Bipartite rigidity.
\newblock {\em Trans. of the Amer. Math. Soc.}, 368(8):5515--5545, 2016.
\newblock \href {http://dx.doi.org/10.1090/tran/6512}
  {\path{doi:10.1090/tran/6512}}.

\bibitem[Ngu10]{Nguyen}
Viet-Hang Nguyen.
\newblock On abstract rigidity matroids.
\newblock {\em SIAM Discrete Math.}, 24:363--369, 2010.
\newblock \href {http://dx.doi.org/10.1137/090762051}
  {\path{doi:10.1137/090762051}}.

\bibitem[NSW21]{NSW}
Anthony Nixon, Bernd Schulze, and Walter Whiteley.
\newblock Rigidity through a projective lens.
\newblock {\em Appl. Sci.}, 2021(11):11946, 2021.
\newblock \href {http://dx.doi.org/10.3390/app112411946}
  {\path{doi:10.3390/app112411946}}.

\bibitem[PS05]{PacStu}
Lior Pachter and Bernd Sturmfels.
\newblock {\em Algebraic Statistics for Computational Biology}.
\newblock Cambridge University Press, 2005.
\newblock \href {http://dx.doi.org/10.1017/CBO9780511610684}
  {\path{doi:10.1017/CBO9780511610684}}.

\bibitem[rgomr10]{IMA}
IMA-ISU research group on~minimum rank.
\newblock Minimum rank of skew-symmetric matrices described by a graph.
\newblock {\em Lin. Alg. App.}, 432:2457--2472, 2010.
\newblock \href {http://dx.doi.org/10.1016/j.laa.2009.10.001}
  {\path{doi:10.1016/j.laa.2009.10.001}}.

\bibitem[Ros14]{Rosen}
Zvi Rosen.
\newblock Computing algebraic matroids.
\newblock Preprint, 2014.
\newblock \href {http://arxiv.org/abs/1403.8148v2} {\path{arXiv:1403.8148v2}}.

\bibitem[Whi90]{Whiteley-split}
Walter Whiteley.
\newblock Vertex splitting in isostatical frameworks.
\newblock {\em Structural Topology}, 16:23--30, 1990.
\newblock URL: \url{http://hdl.handle.net/2099/1055}.

\bibitem[Whi96]{Whiteley}
Walter Whiteley.
\newblock Some matroids from discrete applied geometry.
\newblock In {\em Matroid theory ({S}eattle, {WA}, 1995)}, volume 197 of {\em
  Contemp. Math.}, pages 171--311. Amer. Math. Soc., Providence, RI, 1996.
\newblock \href {http://dx.doi.org/10.1090/conm/197/02540}
  {\path{doi:10.1090/conm/197/02540}}.

\end{thebibliography}

\end{document}